\documentclass{article}%
\usepackage{amssymb}
\usepackage{amsfonts}
\usepackage{amsmath}
\usepackage{fancyhdr}
\usepackage{graphicx}%
\providecommand{\U}[1]{\protect\rule{.1in}{.1in}}
\newtheorem{theorem}{Theorem}

\newtheorem{definition}[theorem]{Definition}

\newtheorem{lemma}[theorem]{Lemma}
\newtheorem{notation}[theorem]{Notation}

\newtheorem{proposition}[theorem]{Proposition}
\newtheorem{remark}[theorem]{Remark}

\newenvironment{proof}[1][Proof]{\noindent\textbf{#1.} }{\ \rule{0.5em}{0.5em}}
\begin{document}

\title{An algebraic result on the topological closure of the set of rational points
on a sphere whose center is non-rational, \textrm{II}}
\author{By Jun-ichi Matsushita}
\date{}
\maketitle

\begin{abstract}
Let $S$ be a sphere in $\mathbb{R}^{n}$ such that $S\cap\mathbb{Q}^{n}%
\neq\emptyset$ and let $\operatorname{Cl}$ denote the closure operator in the
Euclidean topology of $\mathbb{R}^{n}$. If the center of $S$ is in
$\mathbb{Q}^{n}$, then $\operatorname{Cl}(S\cap\mathbb{Q}^{n})$ is $S$, as is
easily proved. If the center of $S$ is \emph{not} in $\mathbb{Q}^{n}$, then
what is $\operatorname{Cl}(S\cap\mathbb{Q}^{n})$? This question, which was
answered partially in the author's paper [Proc. Japan Acad. Ser. A Math. Sci.
\textbf{80 }(2004), no. 7, 146--149], is answered completely in this paper by
representing $\operatorname{Cl}(S\cap\mathbb{Q}^{n})$ in terms of the group of
$\mathbb{Q}$-automorphisms of the algebraic closure of $\mathbb{Q}(\gamma
_{1},\ldots,\gamma_{n})$ in $\mathbb{C}$, where $\gamma_{1},\ldots,\gamma_{n}$
denote the coordinates of the center of $S$.

\end{abstract}

\section{\textbf{Introduction.\label{intr}}}

\begin{notation}
\emph{Let }$\operatorname{Cl}$ \emph{denote the closure operator in the
Euclidean topology of} $\mathbb{R}^{n}$\emph{, and }$\left\vert \cdot
\right\vert $ \emph{(resp. }$\left\langle \cdot,\cdot\right\rangle $\emph{)
the standard Euclidean norm (resp. inner product) in }$\mathbb{R}^{n}$\emph{.
For each }$x=(x_{i})_{1\leqslant i\leqslant n}\in\mathbb{C}^{n}$\emph{, let
}$\Re x$\emph{\ (resp. }$\Im x$\emph{) denote }$(\Re x_{i})_{1\leqslant
i\leqslant n} $\emph{\ (resp. }$(\Im x_{i})_{1\leqslant i\leqslant n}$\emph{),
where }$\Re x_{i}$\emph{\ (resp. }$\Im x_{i}$\emph{) denotes the real (resp.
imaginary) part of }$x_{i}$.
\end{notation}

\footnotetext{2010 \textit{Mathematics Subject Classification}. Primary 14G05;
Secondary 14P25, 12F20, 15A99. \ \ \ \ \
\par
\textit{Key words and Phrases}. sphere, rational point, topological closure,
center, algebraic closure, automorphism group.}Let $S$ be a sphere in
$\mathbb{R}^{n}$ such that $S\cap\mathbb{Q}^{n}\neq\emptyset$ and let us ask
the following question:
\[
\text{What is }\operatorname{Cl}(S\cap\mathbb{Q}^{n})\text{?}%
\]
Our purpose is to answer this question, that is, to generalize the following
easy-to-prove proposition to the case that the center of $S$ is not
necessarily in $\mathbb{Q}^{n}$.

\begin{proposition}
\label{starting}If the center of $S$ is in $\mathbb{Q}^{n}$, then
\begin{equation}
\operatorname{Cl}(S\cap\mathbb{Q}^{n})=S. \tag*{(1)}\label{origin}%
\end{equation}

\end{proposition}

In Matsushita \cite{matsushitaj}, we attained this purpose partially (see
Remark \ref{earlier}). In this paper, we attain it completely. Let
$\gamma=(\gamma_{i})_{1\leqslant i\leqslant n}$ denote the center of $S$ and
let $\Omega$ be the algebraic closure of $\mathbb{Q}(\gamma_{1},\ldots
,\gamma_{n})$ in $\mathbb{C}$ (note that $\gamma$ is in $\Omega^{n}$ as well
as in $\mathbb{R}^{n}$). Let $\Gamma$ be the group of $\mathbb{Q}%
$-automorphisms of $\Omega$ and, for each $\sigma\in\Gamma$, let $\sigma
_{\ast}$ denote the map $\Omega^{n}\rightarrow\Omega^{n}$ given by
$(x_{i})_{1\leqslant i\leqslant n}\mapsto(\sigma(x_{i}))_{1\leqslant
i\leqslant n}$. Fix a point $b=(b_{i})_{1\leqslant i\leqslant n}\in
S\cap\mathbb{Q}^{n}$ and define, for each $a\in\mathbb{R}^{n}$,
\begin{align*}
S_{a}  &  =\{x\in\mathbb{R}^{n}:\left\vert x-a\right\vert =\left\vert
b-a\right\vert \}\\
&  =\text{the sphere through }b\text{ with center }a,\\
\Lambda_{a}  &  =\{x\in\mathbb{R}^{n}:\left\langle a,x-b\right\rangle =0\}\\
&  =\left\{
\begin{array}
[c]{ccc}%
\text{the hyperplane through }b\text{ with normal vector }a & \text{if} &
a\neq0\\
\mathbb{R}^{n} & \text{if} & a=0
\end{array}
\right.
\end{align*}
(note that $S_{\gamma}=S$). Then the result of this paper is stated as follows:

\begin{theorem}
\label{solution}We have
\begin{equation}
\operatorname{Cl}(S\cap\mathbb{Q}^{n})=\bigcap\limits_{\sigma\in\Gamma}%
(S_{\Re\sigma_{\ast}(\gamma)}\cap\Lambda_{\Im\sigma_{\ast}(\gamma)}).
\tag*{(2)}\label{completeanswer}%
\end{equation}

\end{theorem}

Theorem \ref{solution} is certainly a generalization of Proposition
\ref{starting} because if $\gamma$ is in $\mathbb{Q}^{n}$, then $\bigcap
_{\sigma\in\Gamma}(S_{\Re\sigma_{\ast}(\gamma)}\cap\Lambda_{\Im\sigma_{\ast
}(\gamma)})=S_{\Re\gamma}\cap\Lambda_{\Im\gamma}=S_{\gamma}\cap\Lambda
_{0}=S\cap\mathbb{R}^{n}=S$ and hence \ref{completeanswer} coincides with
\ref{origin}. We prove Theorem \ref{solution} in the following sections.

\begin{remark}
\label{earlier}\emph{The result of Matsushita \cite{matsushitaj} is a
generalization (resp. specialization) of Proposition \ref{starting} (resp.
Theorem \ref{solution}) to the case that }the $\mathbb{Q}$-vector subspace of
$\mathbb{R}$ spanned over $\mathbb{Q}$ by $\gamma_{1}-b_{1},\ldots,\gamma
_{n}-b_{n}$ is a field that is Galois over $\mathbb{Q}$\emph{.}
\end{remark}

\section{\textbf{Equivalent transformation of \ref{completeanswer}.
\label{equivalent}}}

In this section, we equivalently transform \ref{completeanswer}. First we
define, for each $a\in\mathbb{R}^{n}$,
\[
\Pi_{a}=\{x\in\mathbb{R}^{n}:\left\langle a-b,x-b\right\rangle =1\}.
\]
Then we have

\begin{proposition}
\label{both}\ref{completeanswer} is equivalent to
\begin{equation}
\operatorname{Cl}(\Pi_{\gamma}\cap\mathbb{Q}^{n})=\bigcap_{\sigma\in\Gamma
}(\Pi_{\Re\sigma_{\ast}(\gamma)}\cap\Lambda_{\Im\sigma_{\ast}(\gamma)}).
\tag*{(3)}\label{rewrite}%
\end{equation}

\end{proposition}

\begin{proof}
As is easily verified, the inversion of $\mathbb{R}^{n}-\{b\}$\ with respect
to the sphere $\{x\in\mathbb{R}^{n}:\left\vert x-b\right\vert =\sqrt{2}\}$
(that is, the map
\begin{equation}
\mathbb{R}^{n}-\{b\}\rightarrow\mathbb{R}^{n}-\{b\};\text{ \ \ \ }x\mapsto
b+\frac{2}{\left\vert x-b\right\vert ^{2}}(x-b) \tag*{(4)}\label{phai}%
\end{equation}
) is a (self-inverse) homeomorphism with respect to the Euclidean topology,
and maps $\Pi_{a}$ (resp. $\Lambda_{a}-\left\{  b\right\}  $) onto
$S_{a}-\left\{  b\right\}  $ (resp. $\Lambda_{a}-\left\{  b\right\}  $) for
each $a\in\mathbb{R}^{n}$, and maps $\mathbb{Q}^{n}-\left\{  b\right\}  $ onto
$\mathbb{Q}^{n}-\left\{  b\right\}  $. Therefore, denoting it by $\varphi$, we
have
\begin{align*}
\text{\ref{rewrite}}  &  \iff\varphi(\operatorname{Cl}(\Pi_{\gamma}%
\cap\mathbb{Q}^{n}))=\varphi\left(  \bigcap_{\sigma\in\Gamma}(\Pi_{\Re
\sigma_{\ast}(\gamma)}\cap\Lambda_{\Im\sigma_{\ast}(\gamma)})\right) \\
&  \iff\operatorname{Cl}(\varphi(\Pi_{\gamma}\cap\mathbb{Q}^{n}%
))-\{b\}=\bigcap_{\sigma\in\Gamma}\varphi(\Pi_{\Re\sigma_{\ast}(\gamma)}%
\cap\Lambda_{\Im\sigma_{\ast}(\gamma)})\\
&  \iff\operatorname{Cl}(S_{\gamma}\cap\mathbb{Q}^{n}-\left\{  b\right\}
)-\{b\}=\bigcap\limits_{\sigma\in\Gamma}(S_{\Re\sigma_{\ast}(\gamma)}%
\cap\Lambda_{\Im\sigma_{\ast}(\gamma)}-\{b\})\\
&  \iff\operatorname{Cl}(S_{\gamma}\cap\mathbb{Q}^{n})-\{b\}=\bigcap
\limits_{\sigma\in\Gamma}(S_{\Re\sigma_{\ast}(\gamma)}\cap\Lambda_{\Im
\sigma_{\ast}(\gamma)})-\{b\}\\
&  \iff\operatorname{Cl}(S_{\gamma}\cap\mathbb{Q}^{n})=\bigcap\limits_{\sigma
\in\Gamma}(S_{\Re\sigma_{\ast}(\gamma)}\cap\Lambda_{\Im\sigma_{\ast}(\gamma
)})\text{,}%
\end{align*}
which proves the proposition since $S_{\gamma}=S$.
\end{proof}

\begin{notation}
\label{innerproduct}\emph{From now on, we extend }$\left\langle \cdot
,\cdot\right\rangle $ \emph{to }$\mathbb{C}^{n}$ \emph{by setting }%
\[
\left\langle x,y\right\rangle =\sum_{i=1}^{n}x_{i}y_{i}%
\]
\emph{for any }$x=(x_{i})_{1\leqslant i\leqslant n},y=(y_{i})_{1\leqslant
i\leqslant n}$\emph{\ in }$\mathbb{C}^{n}$\emph{\ (note that we have }%
\[
\left\langle \sigma_{\ast}(x),\sigma_{\ast}(y)\right\rangle =\sum_{i=1}%
^{n}\sigma(x_{i})\sigma(y_{i})=\emph{\ }\sigma(\sum_{i=1}^{n}x_{i}%
y_{i})=\sigma(\left\langle x,y\right\rangle )
\]
\emph{for any }$x=(x_{i})_{1\leqslant i\leqslant n},y=(y_{i})_{1\leqslant
i\leqslant n}$\emph{\ in }$\Omega^{n}$\emph{\ and any }$\sigma\in\Gamma
$\emph{).}
\end{notation}

Next we define, for each $a\in\Omega^{n}$,%
\[
\Theta_{a}=\{x\in\mathbb{C}^{n}:\forall\sigma\in\Gamma\text{ }\left\langle
\sigma_{\ast}(a)-b,x-b\right\rangle =1\}.
\]
Then we have
\begin{align*}
\Theta_{\gamma}\cap\mathbb{Q}^{n}  &  =\{x\in\mathbb{Q}^{n}:\forall\sigma
\in\Gamma\text{ }\left\langle \sigma_{\ast}(\gamma)-b,x-b\right\rangle =1\}\\
&  =\{x\in\mathbb{Q}^{n}:\forall\sigma\in\Gamma\text{ }\left\langle
\sigma_{\ast}(\gamma-b),\sigma_{\ast}(x-b)\right\rangle =1\}\\
&  =\{x\in\mathbb{Q}^{n}:\forall\sigma\in\Gamma\text{ }\sigma(\left\langle
\gamma-b,x-b\right\rangle )=1\}\\
&  =\{x\in\mathbb{Q}^{n}:\left\langle \gamma-b,x-b\right\rangle =1\}\\
&  =\Pi_{\gamma}\cap\mathbb{Q}^{n},\\
\Theta_{\gamma}\cap\mathbb{R}^{n}  &  =\{x\in\mathbb{R}^{n}:\forall\sigma
\in\Gamma\text{ }\left\langle \sigma_{\ast}(\gamma)-b,x-b\right\rangle =1\}\\
&  =\left\{  x\in\mathbb{R}^{n}:\forall\sigma\in\Gamma\text{ }\left\langle
\Re\sigma_{\ast}(\gamma)+\sqrt{-1}\Im\sigma_{\ast}(\gamma)-b,x-b\right\rangle
=1\right\} \\
&  =\left\{  x\in\mathbb{R}^{n}:\forall\sigma\in\Gamma\text{ }\left\langle
\Re\sigma_{\ast}(\gamma)-b,x-b\right\rangle +\sqrt{-1}\left\langle \Im
\sigma_{\ast}(\gamma),x-b\right\rangle =1\right\} \\
&  =\left\{  x\in\mathbb{R}^{n}:\forall\sigma\in\Gamma\text{ }(\left\langle
\Re\sigma_{\ast}(\gamma)-b,x-b\right\rangle =1,\text{ }\left\langle \Im
\sigma_{\ast}(\gamma),x-b\right\rangle =0)\right\} \\
&  =\bigcap_{\sigma\in\Gamma}(\Pi_{\Re\sigma_{\ast}(\gamma)}\cap\Lambda
_{\Im\sigma_{\ast}(\gamma)})
\end{align*}
and hence

\begin{proposition}
\label{both2}\ref{rewrite} is equivalent to
\begin{equation}
\operatorname{Cl}(\Theta_{\gamma}\cap\mathbb{Q}^{n})=\Theta_{\gamma}%
\cap\mathbb{R}^{n}. \tag*{(5)}\label{paisolution2}%
\end{equation}

\end{proposition}

By Proposition \ref{both} and Proposition \ref{both2}, we have that
\ref{completeanswer} is equivalent to \ref{paisolution2}.

\section{\textbf{Rationality for affine subspaces of }$\mathbb{C}^{n}%
$\textbf{.\label{definition}}}

\begin{notation}
\emph{In this section and the next section, we use the following notation:}

\emph{(a) For any field }$F$ \emph{and any subset }$M$ \emph{of }%
$\{1,\ldots,n\}$\emph{, }$F^{M}$ \emph{denotes the set of families of elements
of }$F$\emph{\ indexed by }$M$\emph{ (regarded as a }$F$\emph{-vector space by
defining addition and scalar multiplication as }$\left\{  x_{i}\right\}
_{i\in M}+\left\{  x_{i}^{\prime}\right\}  _{i\in M}=\left\{  x_{i}%
+x_{i}^{\prime}\right\}  _{i\in M}$\emph{\ and }$c\left\{  x_{i}\right\}
_{i\in M}=\left\{  cx_{i}\right\}  _{i\in M}$\emph{).}

\emph{(b) For any subset }$M$ \emph{of }$\{1,\ldots,n\}$\emph{, }$P_{M}%
$\emph{\ denotes the linear surjection }$\mathbb{C}^{n}\rightarrow
\mathbb{C}^{M}$\emph{\ given by }$(x_{i})_{1\leqslant i\leqslant n}%
\mapsto\left\{  x_{i}\right\}  _{i\in M}$.
\end{notation}

Let $U$ be a non-empty affine subspace of $\mathbb{C}^{n}$ and let $K$ be a
subfield of $\mathbb{C}$. In this section, we introduce the notion of
\emph{rationality} of $U$ over $K$, which plays important roles in the next
section. First we define three terms on $U$.

\begin{definition}
\label{frame}\emph{For any }$\beta_{0},\ldots,\beta_{m}\in\mathbb{C}^{n}%
$\emph{\ such that }%
\[
U=\beta_{0}+(\mathbb{C}\beta_{1}\oplus\cdots\oplus\mathbb{C}\beta_{m}),
\]
$(\beta_{0};\beta_{1},\ldots,\beta_{m})$\emph{\ is called a }frame\emph{\ for
}$U$\emph{.}
\end{definition}

\begin{definition}
\emph{A system of linear equations in }$x_{1},\ldots,x_{n}$ \emph{with
coefficients in }$\mathbb{C}$ \emph{is called a }defining system\emph{\ for\ }%
$U$ \emph{if it is a necessary and sufficient condition for }$(x_{i}%
)_{1\leqslant i\leqslant n}\in\mathbb{C}^{n}$ \emph{to be in }$U$\emph{, that
is, if }$U$ \emph{is the set of }$(x_{i})_{1\leqslant i\leqslant n}%
\in\mathbb{C}^{n}$\emph{\ satisfying it\ (note that, since }$U$\emph{ is
non-empty, any defining system for }$U$\emph{ is consistent).}
\end{definition}

\begin{definition}
\emph{A pair }$(M,f)$\emph{\ of a subset }$M$\emph{\ of }$\left\{
1,\ldots,n\right\}  $\emph{\ and an affine map }$f:$\emph{\ }$\mathbb{C}%
^{M}\rightarrow\mathbb{C}^{M^{c}}$\emph{\ is called a }defining
pair\emph{\ for }$U$\emph{\ if }%
\[
P_{M^{c}}(x)=f(P_{M}(x))
\]
\emph{is a necessary and sufficient condition for }$x\in\mathbb{C}^{n}$
\emph{to be in }$U$\emph{, that is, if }%
\begin{equation}
\left\{  x_{i}\right\}  _{i\in M^{c}}=f(\left\{  x_{i}\right\}  _{i\in M})
\tag*{(6)}\label{system}%
\end{equation}
\emph{is a defining system for }$U$\emph{. }
\end{definition}

\begin{remark}
\label{existence}\emph{As is clear from elementary linear algebra, there
exists a defining system for }$U$\emph{\ in reduced echelon form, which, when
solved for the leading variables in terms of the free variables, can be
written as \ref{system} for some subset }$M$\emph{\ of }$\left\{
1,\ldots,n\right\}  $\emph{\ and some affine map }$f:$\emph{\ }$\mathbb{C}%
^{M}\rightarrow\mathbb{C}^{M^{c}}$\emph{. Therefore there exist a subset }%
$M$\emph{\ of }$\left\{  1,\ldots,n\right\}  $\emph{\ and an affine map }%
$f:$\emph{\ }$\mathbb{C}^{M}\rightarrow\mathbb{C}^{M^{c}}$\emph{\ such that
\ref{system} is a defining system for }$U$\emph{, that is, there exists a
defining pair for }$U$\emph{.}
\end{remark}

Next we note that $U\cap K^{n}$ is clearly an affine subspace of $K^{n}$ and
prove two lemmas.

\begin{lemma}
\label{fn}Let $(\beta_{0};\beta_{1},\ldots,\beta_{m})$ be a frame for $U$.
Then we have
\[
\beta_{0},\ldots,\beta_{m}\in K^{n}\implies U\cap K^{n}=\beta_{0}+(K\beta
_{1}\oplus\cdots\oplus K\beta_{m}).
\]

\end{lemma}

\begin{proof}
Since%
\[
\beta_{0},\ldots,\beta_{m}\in K^{n}\implies U\cap K^{n}\supset\beta
_{0}+(K\beta_{1}\oplus\cdots\oplus K\beta_{m})
\]
is clear, we need only show%
\[
\beta_{0},\ldots,\beta_{m}\in K^{n}\implies U\cap K^{n}\subset\beta
_{0}+(K\beta_{1}\oplus\cdots\oplus K\beta_{m}).
\]
This holds because, letting $p$ be a $K$-linear projection from the $K$-vector
space $\mathbb{C}$ onto its subspace $K$ and defining the map $p_{\ast
}:\mathbb{C}^{n}\rightarrow K^{n}$ by $p_{\ast}((x_{i})_{1\leqslant i\leqslant
n})=(p(x_{i}))_{1\leqslant i\leqslant n}$, we have
\begin{align*}
\beta_{0},\ldots,\beta_{m}\in K^{n}\implies p_{\ast}(U)  &  =p_{\ast}%
(\beta_{0}+(\mathbb{C}\beta_{1}\oplus\cdots\oplus\mathbb{C}\beta_{m}))\\
&  =p_{\ast}(\beta_{0})+(p(\mathbb{C)}\beta_{1}\oplus\cdots\oplus
p(\mathbb{C)}\beta_{m})\\
&  =\beta_{0}+(K\beta_{1}\oplus\cdots\oplus K\beta_{m})
\end{align*}
and $U\cap K^{n}=p_{\ast}(U\cap K^{n})\subset p_{\ast}(U)$.
\end{proof}

\begin{lemma}
\label{dimrelation}Let $(M,f)$ be a defining pair for $U$. Then the following hold:

\emph{(a) }$U$ is isomorphic to $\mathbb{C}^{M}$.

\emph{(b) }$U\cap K^{n}$ is isomorphic to the affine subspace $K^{M}\cap
f^{-1}(K^{M^{c}})$ of $K^{M}$.
\end{lemma}

\begin{proof}
Since any $x\in K^{n}$ satisfies $P_{M}(x)\in K^{M}$ and any $x,x^{\prime}\in
U$ satisfy
\begin{align*}
x=x^{\prime}  &  \iff(P_{M}(x)=P_{M}(x^{\prime}),\text{ }P_{M^{c}}%
(x)=P_{M^{c}}(x^{\prime}))\\
&  \iff(P_{M}(x)=P_{M}(x^{\prime}),\text{ }f(P_{M}(x))=f(P_{M}(x^{\prime})))\\
&  \iff P_{M}(x)=P_{M}(x^{\prime}),
\end{align*}
$P_{M}$ gives a one-to-one map $U\cap K^{n}\rightarrow K^{M}$. This map is
clearly affine, and its image is $K^{M}\cap f^{-1}(K^{M^{c}})$ because
\begin{align*}
&  P_{M}(U\cap K^{n})\\
&  =P_{M}(\{x\in K^{n}:P_{M^{c}}(x)=f(P_{M}(x))\})\\
&  =P_{M}(\{x\in\mathbb{C}^{n}:P_{M}(x)\in K^{M},\text{ }P_{M^{c}}(x)\in
K^{M^{c}},\text{ }P_{M^{c}}(x)=f(P_{M}(x))\})\\
&  =P_{M}(\{x\in\mathbb{C}^{n}:P_{M}(x)\in K^{M},\text{ }f(P_{M}(x))\in
K^{M^{c}},\text{ }P_{M^{c}}(x)=f(P_{M}(x))\})\\
&  =\{y\in\mathbb{C}^{M}:y\in K^{M},\text{ }f(y)\in K^{M^{c}}\}\\
&  =K^{M}\cap f^{-1}(K^{M^{c}}).
\end{align*}
Therefore there exists a one-to-one affine map $U\cap K^{n}\rightarrow K^{M}%
$\ whose image is $K^{M}\cap f^{-1}(K^{M^{c}})$, that is, (b) holds. (a) is
the case $K=\mathbb{C}$ of (b).
\end{proof}

Under these preparations, we prove a proposition and define the notion of
\emph{rationality }of $U$ over $K$.

\begin{proposition}
\label{rationality}The following conditions are equivalent:

\emph{(a) }There exists a frame $(\beta_{0};\beta_{1},\ldots,\beta_{m})$ for
$U$ satisfying $\beta_{0},\ldots,\beta_{m}\in K^{n}$.

\emph{(b) }There exists a defining system for $U$ with coefficients in $K$.

\emph{(c) }There exists a defining pair $(M,f)$ for $U$ satisfying
$f(K^{M})\subset K^{M^{c}}$.

\emph{(d) }Every defining pair $(M,f)$ for $U$ satisfies $f(K^{M})\subset
K^{M^{c}}$.\emph{\ }
\end{proposition}

\begin{proof}
We prove this by showing the equivalence of (a), (b), (c), (d), and
\begin{equation}
\dim_{K}(U\cap K^{n})=\dim_{\mathbb{C}}U. \tag*{(7)}\label{dimkdimc}%
\end{equation}

(d)$\implies$(c): This is clear since, as was shown in Remark \ref{existence},
there exists a defining pair for $U$.

(c)$\implies$(b): This holds because, as is easily seen, every defining pair
$(M,f)$ for $U$ satisfies that if $f(K^{M})\subset K^{M^{c}}$, then
\ref{system} is a defining system for $U$ with coefficients in $K$.

(b)$\implies$(a): This is readily seen from the faithfully flatness of
$\mathbb{C}$ over $K$.

(a)$\implies$\ref{dimkdimc}: This holds because, by Lemma \ref{fn}, every
frame $(\beta_{0};\beta_{1},\ldots,\beta_{m})$ for $U$ satisfies
\begin{align*}
\beta_{0},\ldots,\beta_{m}\in K^{n}\implies\dim_{K}(U\cap K^{n})  &  =\dim
_{K}(\beta_{0}+(K\beta_{1}\oplus\cdots\oplus K\beta_{m}))\\
&  =\dim_{\mathbb{C}}(\beta_{0}+(\mathbb{C}\beta_{1}\oplus\cdots
\oplus\mathbb{C}\beta_{m}))\\
&  =\dim_{\mathbb{C}}U.
\end{align*}

\ref{dimkdimc}$\implies$(d): This holds because, by Lemma \ref{dimrelation},
every defining pair $(M,f)$ for $U$ satisfies
\begin{align*}
\text{\ref{dimkdimc}}  &  \iff\dim_{K}(K^{M}\cap f^{-1}(K^{M^{c}}%
))=\dim_{\mathbb{C}}\mathbb{C}^{M}\\
&  \iff\dim_{K}(K^{M}\cap f^{-1}(K^{M^{c}}))=\dim_{K}K^{M}\\
&  \iff K^{M}\cap f^{-1}(K^{M^{c}})=K^{M}\\
&  \iff K^{M}\subset f^{-1}(K^{M^{c}})\\
&  \iff f(K^{M})\subset K^{M^{c}}.
\end{align*}

\end{proof}

\begin{definition}
\label{rationality2}$U$\emph{\ is said to be }rational\ \emph{over\ }%
$K$\emph{\ if the equivalent conditions of Proposition \ref{rationality} hold
(compare this with Bourbaki \cite[p. 318, Definition 2]{bourbakili}).}
\end{definition}

\section{Proof of \ref{paisolution2}.\label{proof}}

\begin{notation}
\emph{In this section, we use the following notation:}

\emph{(a) For any subset }$M$ \emph{of }$\{1,\ldots,n\}$\emph{, }$I_{M}%
$\emph{\ denotes the bijection }$\mathbb{C}^{n}\rightarrow\mathbb{C}^{M}%
\times\mathbb{C}^{M^{c}}$\emph{\ given by }$(x_{i})_{1\leqslant i\leqslant
n}\mapsto(\left\{  x_{i}\right\}  _{i\in M},\left\{  x_{i}\right\}  _{i\in
M^{c}})$\emph{, that is, }$x\mapsto(P_{M}(x),P_{M^{c}}(x))$\emph{.}

\emph{(b) For any }$\tau\in\Gamma$\emph{\ and any subset }$M$\emph{\ of
}$\{1,\ldots,n\}$\emph{,} $\tau_{M}$ \emph{denotes the bijection }$\Omega
^{M}\rightarrow\Omega^{M}$\emph{\ given by }$\left\{  x_{i}\right\}  _{i\in
M}\mapsto\left\{  \tau(x_{i})\right\}  _{i\in M}$.

\emph{(c) For any }$\tau\in\Gamma$\emph{\ and any subset }$M$\emph{\ of
}$\{1,\ldots,n\}$\emph{, }$\tau_{M}^{-1}$\emph{ denotes }$(\tau_{M})^{-1}%
$\emph{, that is, }$(\tau^{-1})_{M}$\emph{. }

\emph{(d) For any }$\tau\in\Gamma$\emph{, }$\tau_{\ast}^{-1}$\emph{ denotes}
$(\tau_{\ast})^{-1}$\emph{, that is, }$(\tau^{-1})_{\ast}$\emph{.}
\end{notation}

In this section, we complete our proof of Theorem \ref{solution} by proving
that \ref{paisolution2}, which was shown to be equivalent to
\ref{completeanswer} in Section \ref{equivalent}, holds. Since the case
$\Theta_{\gamma}=\emptyset$ is trivial, we restrict ourselves to the case
$\Theta_{\gamma}\neq\emptyset$, in which $\Theta_{\gamma}$ is (since it is
clearly an affine subspace of $\mathbb{C}^{n}$) a non-empty affine subspace of
$\mathbb{C}^{n}$, that is, an example of $U$. First we prove two lemmas.

\begin{lemma}
\label{graph}If $U$ is rational over $\Omega$, every defining pair $(M,f)$ for
$U$ satisfies
\begin{equation}
I_{M}(\tau_{\ast}(U\cap\Omega^{n}))=\left\{  \left(  y,\tau_{M^{c}}(f(\tau
_{M}^{-1}(y)))\right)  :y\in\Omega^{M}\right\}  \tag*{(8)}\label{graph3}%
\end{equation}
for every $\tau\in\Gamma$.
\end{lemma}

\begin{proof}
For every $\tau\in\Gamma$, every defining pair $(M,f)$ for $U$ satisfies
\begin{align*}
&  I_{M}(\tau_{\ast}(U\cap\Omega^{n}))\\
&  =I_{M}(\tau_{\ast}(\{x\in\Omega^{n}:P_{M^{c}}(x)=f(P_{M}(x))\}))\\
&  =I_{M}(\{x\in\Omega^{n}:P_{M^{c}}(\tau_{\ast}^{-1}(x))=f(P_{M}(\tau_{\ast
}^{-1}(x)))\})\\
&  =I_{M}(\{x\in\Omega^{n}:\tau_{M^{c}}^{-1}(P_{M^{c}}(x))=f(\tau_{M}%
^{-1}(P_{M}(x)))\})\\
&  =I_{M}(\{x\in\mathbb{C}^{n}:P_{M}(x)\in\Omega^{M},\text{ }P_{M^{c}}%
(x)\in\Omega^{M^{c}},\text{ }\tau_{M^{c}}^{-1}(P_{M^{c}}(x))=f(\tau_{M}%
^{-1}(P_{M}(x)))\})\\
&  =\{(y,z)\in\mathbb{C}^{M}\times\mathbb{C}^{M^{c}}:y\in\Omega^{M},\text{
}z\in\Omega^{M^{c}},\text{ }\tau_{M^{c}}^{-1}(z)=f(\tau_{M}^{-1}(y))\}\\
&  =\{(y,z)\in\mathbb{C}^{M}\times\mathbb{C}^{M^{c}}:y\in\Omega^{M},\text{
}f(\tau_{M}^{-1}(y))\in\Omega^{M^{c}},\text{ }z=\tau_{M^{c}}(f(\tau_{M}%
^{-1}(y)))\}\\
&  =\left\{  \left(  y,\tau_{M^{c}}(f(\tau_{M}^{-1}(y)))\right)  :y\in
\Omega^{M},\text{ }f(\tau_{M}^{-1}(y))\in\Omega^{M^{c}}\right\}
\end{align*}
and hence
\begin{align*}
\text{\ref{graph3}}  &  \impliedby\forall y\in\Omega^{M}\ \ f(\tau_{M}%
^{-1}(y))\in\Omega^{M^{c}}\\
&  \iff\forall y\in\tau_{M}^{-1}(\Omega^{M})\ \ f(y)\in\Omega^{M^{c}}\\
&  \iff\forall y\in\Omega^{M}\ \ f(y)\in\Omega^{M^{c}}\text{ \ \ (}%
\because\tau_{M}^{-1}(\Omega^{M})=\Omega^{M}\text{)}\\
&  \iff f(\Omega^{M})\subset\Omega^{M^{c}}.
\end{align*}
Therefore every defining pair $(M,f)$ for $U$ satisfies \ref{graph3} for every
$\tau\in\Gamma$ if every defining pair $(M,f)$ for $U$ satisfies $f(\Omega
^{M})\subset\Omega^{M^{c}}$, that is, if $U$ is rational over $\Omega$\emph{.
}
\end{proof}

\begin{lemma}
\label{invariant}The set of invariants of $\Gamma\ $is equal to $\mathbb{Q}$.
\end{lemma}

\begin{proof}
By Bourbaki \cite[Chapter \textrm{V}, p. 112, Proposition 10]{bourbakifi}, the
set of invariants of the group of $\mathbb{Q}$-automorphisms of any
algebraically closed extension of $\mathbb{Q}$ is equal to $\mathbb{Q}$. Since
$\Omega$, of which $\Gamma$ is the group of $\mathbb{Q}$-automorphisms, is an
algebraically closed extension of $\mathbb{Q}$, this implies the lemma.
\end{proof}

Using Lemma \ref{fn}\ and these two lemmas, we prove

\begin{proposition}
\label{wqw}If $U$ is rational over $\Omega$ and
\begin{equation}
\forall\tau\in\Gamma\text{ \ }\tau_{\ast}(U\cap\Omega^{n})=U\cap\Omega^{n},
\tag*{(9)}\label{invariance}%
\end{equation}
then
\begin{equation}
\operatorname{Cl}(U\cap\mathbb{Q}^{n})=U\cap\mathbb{R}^{n}. \tag*{(10)}%
\label{wdash}%
\end{equation}

\end{proposition}

\begin{proof}
We prove this by showing that
\begin{equation}
\text{\ref{wdash} holds if\emph{\ }}U\text{\emph{\ }is rational over
}\mathbb{Q\label{rationalq}} \tag*{(11)}%
\end{equation}
and that
\begin{equation}
U\text{ is rational over }\mathbb{Q}\text{ if }U\text{ is rational over
}\Omega\text{ and\emph{\ }\ref{invariance} holds.\label{rationalq2}}
\tag*{(12)}%
\end{equation}

\ref{rationalq} is shown as follows: By Lemma \ref{fn}, every frame
$(\beta_{0};\beta_{1},\ldots,\beta_{m})$ for $U$ satisfies
\begin{align*}
\beta_{0},\ldots,\beta_{m}\in\mathbb{Q}^{n}  &  \implies U\cap\mathbb{Q}%
^{n}=\beta_{0}+(\mathbb{Q}\beta_{1}\oplus\cdots\oplus\mathbb{Q}\beta_{m}),\\
\beta_{0},\ldots,\beta_{m}\in\mathbb{R}^{n}  &  \implies U\cap\mathbb{R}%
^{n}=\beta_{0}+(\mathbb{R}\beta_{1}\oplus\cdots\oplus\mathbb{R}\beta_{m})
\end{align*}
and hence
\begin{align*}
\beta_{0},\ldots,\beta_{m}\in\mathbb{Q}^{n}\implies\operatorname{Cl}%
(U\cap\mathbb{Q}^{n})  &  =\operatorname{Cl}(\beta_{0}+(\mathbb{Q}\beta
_{1}\oplus\cdots\oplus\mathbb{Q}\beta_{m}))\\
&  =\beta_{0}+(\mathbb{R}\beta_{1}\oplus\cdots\oplus\mathbb{R}\beta_{m})\\
&  =U\cap\mathbb{R}^{n}\text{.}%
\end{align*}
Therefore \ref{wdash} holds if there exists a frame $(\beta_{0};\beta
_{1},\ldots,\beta_{m})$ for $U$ satisfying $\beta_{0},\ldots,\beta_{m}%
\in\mathbb{Q}^{n}$, that is, if $U$\emph{\ }is rational over $\mathbb{Q}$.

\ref{rationalq2} is shown as follows: If $U$ is rational over $\Omega$, then,
by Lemma \ref{graph}, every defining pair $(M,f)$ for $U$ satisfies
\begin{align}
\text{\ref{invariance}}  &  \iff\forall\tau\in\Gamma\text{ \ }I_{M}(\tau
_{\ast}(U\cap\Omega^{n}))=I_{M}(U\cap\Omega^{n})\nonumber\\
&  \iff\forall\tau\in\Gamma\text{ \ }\left\{  \left(  y,\tau_{M^{c}}%
(f(\tau_{M}^{-1}(y)))\right)  :y\in\Omega^{M}\right\}  =\{\left(
y,f(y)\right)  :y\in\Omega^{M}\}\nonumber\\
&  \iff\forall\tau\in\Gamma\text{ \ }\forall y\in\Omega^{M}\text{ \ }%
\tau_{M^{c}}(f(\tau_{M}^{-1}(y)))=f(y)\nonumber\\
&  \implies\forall\tau\in\Gamma\text{ \ }\forall y\in\mathbb{Q}^{M}\text{
\ }\tau_{M^{c}}(f(\tau_{M}^{-1}(y)))=f(y)\nonumber\\
&  \iff\forall\tau\in\Gamma\text{ \ }\forall y\in\mathbb{Q}^{M}\text{ \ }%
\tau_{M^{c}}(f(y))=f(y)\nonumber\\
&  \iff\forall y\in\mathbb{Q}^{M}\ \text{\ }f(y)\in\mathbb{Q}^{M^{c}}\text{
\ \ (by Lemma \ref{invariant})}\nonumber\\
&  \iff f(\mathbb{Q}^{M})\subset\mathbb{Q}^{M^{c}}. \tag*{(13)}\label{taufyfy}%
\end{align}
Therefore if $U$ is rational over $\Omega$ and\emph{\ }\ref{invariance} holds,
then every defining pair $(M,f)$ for $U$ satisfies $f(\mathbb{Q}^{M}%
)\subset\mathbb{Q}^{M^{c}}$, that is, $U$\emph{\ }is rational over
$\mathbb{Q}$.
\end{proof}

\begin{remark}
\emph{Noting that }$f$\emph{\ is an affine map and using the obvious fact that
}each element of $\Omega^{M}$\emph{\ }can be written as an affine combination
(over $\Omega$) of elements of $\mathbb{Q}^{M}$\emph{, we can easily justify
replacing `}$\implies$\emph{'\ in \ref{taufyfy} by `}$\iff$\emph{', which
changes the first sentence of the proof of \ref{rationalq2} to imply not only
that }if $U$ is rational over $\Omega$ and\emph{\ }\ref{invariance} holds,
then every defining pair $(M,f)$ for $U$ satisfies $f(\mathbb{Q}^{M}%
)\subset\mathbb{Q}^{M^{c}}$, that is, $U$\emph{\ }is rational over
$\mathbb{Q}$\emph{\ but also that }if $U$ is rational over $\Omega$, then
\ref{invariance} holds if there exists a defining pair $(M,f)$ for $U$
satisfying\emph{\ }$f(\mathbb{Q}^{M})\subset\mathbb{Q}^{M^{c}}$, that is,
if\emph{\ }$U$\emph{\ }is rational over\emph{\ }$\mathbb{Q}$\emph{, which and
the obvious fact that }$U$ is rational over $\Omega$\emph{\ }if\emph{\ }%
$U$\emph{\ }is rational over\emph{\ }$\mathbb{Q}$\emph{\ imply the converse of
\ref{rationalq2}. Hence we have not only \ref{rationalq2} but also its
converse, that is, }$U$ is rational over $\mathbb{Q}$ if and only if $U$ is
rational over $\Omega$ and\emph{\ }\ref{invariance} holds\emph{\ (compare this
with Bourbaki \cite[p. 324, Theorem 1 (i)]{bourbakili}).}
\end{remark}

Now we prove that $U=\Theta_{\gamma}$ satisfies the rationality of $U$ over
$\Omega$ and \ref{invariance}, that is, the following holds, which, by
Proposition \ref{wqw}, implies that $U=\Theta_{\gamma}$ satisfies \ref{wdash},
that is, \ref{paisolution2} holds.

\begin{proposition}
\label{wgamma}$\Theta_{\gamma}$ is rational over $\Omega$ and%
\begin{equation}
\forall\tau\in\Gamma\text{ \ }\tau_{\ast}(\Theta_{\gamma}\cap\Omega
^{n})=\Theta_{\gamma}\cap\Omega^{n}. \tag*{(14)}\label{symmetry}%
\end{equation}

\end{proposition}

\begin{proof}
Since
\begin{align*}
\Theta_{\gamma}  &  =\{x\in\mathbb{C}^{n}:\forall\sigma\in\Gamma\text{
}\left\langle \sigma_{\ast}(\gamma)-b,x-b\right\rangle =1\}\\
&  =\{x\in\mathbb{C}^{n}:\forall\sigma\in\Gamma\text{ }\left\langle
\sigma_{\ast}(\gamma)-b,x\right\rangle =1+\left\langle \sigma_{\ast}%
(\gamma)-b,b\right\rangle \}\\
&  =\left\{  (x_{i})_{1\leqslant i\leqslant n}\in\mathbb{C}^{n}:\forall
\sigma\in\Gamma\text{ }\sum_{i=1}^{n}(\sigma(\gamma_{i})-b_{i})x_{i}%
=1+\left\langle \sigma_{\ast}(\gamma)-b,b\right\rangle \right\}  ,
\end{align*}
a defining system for $\Theta_{\gamma}$ is
\[
\forall\sigma\in\Gamma\text{ \ }\sum_{i=1}^{n}(\sigma(\gamma_{i})-b_{i}%
)x_{i}=1+\left\langle \sigma_{\ast}(\gamma)-b,b\right\rangle ,
\]
which is with coefficients in $\Omega$. Hence there exists a defining system
for $\Theta_{\gamma}$ with coefficients in $\Omega$, that is, $\Theta_{\gamma
}$ is rational over $\Omega$. \ref{symmetry} holds because
\begin{align*}
\forall\tau\in\Gamma\text{ \ }\tau_{\ast}(\Theta_{\gamma}\cap\Omega^{n})  &
=\tau_{\ast}(\{x\in\Omega^{n}:\forall\sigma\in\Gamma\text{ }\left\langle
\sigma_{\ast}(\gamma)-b,x-b\right\rangle =1\})\\
&  =\left\{  x\in\Omega^{n}:\forall\sigma\in\Gamma\text{ }\left\langle
\sigma_{\ast}(\gamma)-b,\tau_{\ast}^{-1}(x)-b\right\rangle =1\right\} \\
&  =\left\{  x\in\Omega^{n}:\forall\sigma\in\Gamma\text{ }\left\langle
\tau_{\ast}^{-1}((\tau\sigma)_{\ast}(\gamma)-b),\tau_{\ast}^{-1}%
(x-b)\right\rangle =1\right\} \\
&  =\{x\in\Omega^{n}:\forall\sigma\in\Gamma\text{ }\tau^{-1}(\left\langle
(\tau\sigma)_{\ast}(\gamma)-b,x-b\right\rangle )=1\}\\
&  =\{x\in\Omega^{n}:\forall\sigma\in\Gamma\text{ }\left\langle (\tau
\sigma)_{\ast}(\gamma)-b,x-b\right\rangle =1\}\\
&  =\{x\in\Omega^{n}:\forall\sigma\in\tau\Gamma\text{ }\left\langle
\sigma_{\ast}(\gamma)-b,x-b\right\rangle =1\}\\
&  =\{x\in\Omega^{n}:\forall\sigma\in\Gamma\text{ }\left\langle \sigma_{\ast
}(\gamma)-b,x-b\right\rangle =1\}\text{ \ \ (}\because\tau\Gamma
=\Gamma\text{)}\\
&  =\Theta_{\gamma}\cap\Omega^{n}.
\end{align*}

\end{proof}

\begin{remark}
\label{rtoc}\emph{Let }$a=(a_{i})_{1\leqslant i\leqslant n}$\emph{\ be an
element of }$\Omega^{n}$\emph{. Then, replacing }$\gamma$\emph{\ (resp.
}$\gamma_{i}$\emph{) by }$a$\emph{\ (resp. }$a_{i}$\emph{) everywhere in our
proof of \ref{paisolution2}, we obtain }%
\[
\operatorname{Cl}(\Theta_{a}\cap\mathbb{Q}^{n})=\Theta_{a}\cap\mathbb{R}^{n},
\]
\emph{which, by modifying the proof of Proposition \ref{both2}, is shown to be
equivalent to \ }%
\[
\operatorname{Cl}(\Pi_{\Re a}\cap\Lambda_{\Im a}\cap\mathbb{Q}^{n}%
)=\bigcap_{\sigma\in\Gamma}(\Pi_{\Re\sigma_{\ast}(a)}\cap\Lambda_{\Im
\sigma_{\ast}(a)}),
\]
\emph{which, by modifying the proof of Proposition \ref{both}, is shown to be
equivalent to }%
\[
\operatorname{Cl}(S_{\Re a}\cap\Lambda_{\Im a}\cap\mathbb{Q}^{n}%
)=\bigcap_{\sigma\in\Gamma}(S_{\Re\sigma_{\ast}(a)}\cap\Lambda_{\Im
\sigma_{\ast}(a)}).
\]
\emph{\ref{completeanswer} is the case }$a=\gamma$\emph{\ of this last
formula.}
\end{remark}

\begin{remark}
\emph{Let }$q$ \emph{be a non-degenerate definite quadratic form }%
$\mathbb{R}^{n}\rightarrow\mathbb{R}$\emph{\ such that }$q(\mathbb{Q}%
^{n})\subset\mathbb{Q}$\emph{, and }$B$ \emph{the polar form of }$q$\emph{,
that is, the symmetric bilinear form }$\mathbb{R}^{n}\times\mathbb{R}%
^{n}\rightarrow\mathbb{R}$\emph{ that is defined by }$B(x,y)\equiv\frac{1}%
{2}(q(x+y)-q(x)-q(y))$\emph{ (and hence satisfies }$B(x,x)\equiv q(x)$\emph{),
and }$L$ \emph{the linear map }$\mathbb{R}^{n}\rightarrow\mathbb{R}^{n}%
$\emph{\ such that }$B(x,y)\equiv\left\langle L(x),y\right\rangle
$\emph{\ (note that an example of }$q$\emph{ is }$q(x)\equiv\left\vert
x\right\vert ^{2}$\emph{, which implies }$B(x,y)\equiv\left\langle
x,y\right\rangle $\emph{ and hence }$L(x)\equiv x$\emph{). Then, as is easily
proved, the map }%
\begin{equation}
\mathbb{R}^{n}-\{b\}\rightarrow\mathbb{R}^{n}-\{b\};\text{ \ \ \ }x\mapsto
b+\frac{2}{q(L^{-1}(x-b))}L^{-1}(x-b) \tag*{(15)}\label{phaiq}%
\end{equation}
\emph{is well-defined, and is a homeomorphism with respect to the Euclidean
topology (with the inverse }$x\mapsto b+\frac{2}{q(x-b)}L(x-b)$\emph{), and
maps }$\Pi_{a}$\emph{\ (resp. }$\Lambda_{a}-\left\{  b\right\}  $\emph{) onto
}%
\begin{equation}
\{x\in\mathbb{R}^{n}-\left\{  b\right\}  :q\left(  x-a\right)  =q\left(
b-a\right)  \}\emph{\ }\text{\emph{(resp. }}\{x\in\mathbb{R}^{n}-\left\{
b\right\}  :B(a,x-b)=0\}\text{\emph{)\label{sdash}}} \tag*{(16)}%
\end{equation}
\emph{for each }$a\in\mathbb{R}^{n}$\emph{, and maps }$\mathbb{Q}^{n}-\left\{
b\right\}  $\emph{\ onto }$\mathbb{Q}^{n}-\left\{  b\right\}  $\emph{ (note
that if }$q(x)\equiv\left\vert x\right\vert ^{2}$\emph{ and hence
}$B(x,y)\equiv\left\langle x,y\right\rangle $\emph{, }$L(x)\equiv x$\emph{,
then \ref{phaiq} coincides with \ref{phai} and \ref{sdash} coincides with
}$S_{a}-\left\{  b\right\}  $\emph{\ (resp. }$\Lambda_{a}-\left\{  b\right\}
$\emph{) for each }$a\in\mathbb{R}^{n}$\emph{). Therefore, denoting it by
}$\varphi^{\prime}$\emph{ and denoting }%
\[
\{x\in\mathbb{R}^{n}:q\left(  x-a\right)  =q\left(  b-a\right)
\}\text{\emph{\ (resp. }}\{x\in\mathbb{R}^{n}:B(a,x-b)=0\}\text{\emph{)}}%
\]
\emph{by }$S_{a}^{\prime}$\emph{\ (resp. }$\Lambda_{a}^{\prime}$\emph{) for
each }$a\in\mathbb{R}^{n}$\emph{, we have }%
\begin{align*}
\text{\emph{\ref{rewrite}}}  &  \iff\varphi^{\prime}(\operatorname{Cl}%
(\Pi_{\gamma}\cap\mathbb{Q}^{n}))=\varphi^{\prime}\left(  \bigcap_{\sigma
\in\Gamma}(\Pi_{\Re\sigma_{\ast}(\gamma)}\cap\Lambda_{\Im\sigma_{\ast}%
(\gamma)})\right) \\
&  \iff\operatorname{Cl}(\varphi^{\prime}(\Pi_{\gamma}\cap\mathbb{Q}%
^{n}))-\{b\}=\bigcap_{\sigma\in\Gamma}\varphi^{\prime}(\Pi_{\Re\sigma_{\ast
}(\gamma)}\cap\Lambda_{\Im\sigma_{\ast}(\gamma)})\\
&  \iff\operatorname{Cl}(S_{\gamma}^{\prime}\cap\mathbb{Q}^{n}-\left\{
b\right\}  )-\{b\}=\bigcap\limits_{\sigma\in\Gamma}(S_{\Re\sigma_{\ast}%
(\gamma)}^{\prime}\cap\Lambda_{\Im\sigma_{\ast}(\gamma)}^{\prime}-\{b\})\\
&  \iff\operatorname{Cl}(S_{\gamma}^{\prime}\cap\mathbb{Q}^{n})-\{b\}=\bigcap
\limits_{\sigma\in\Gamma}(S_{\Re\sigma_{\ast}(\gamma)}^{\prime}\cap
\Lambda_{\Im\sigma_{\ast}(\gamma)}^{\prime})-\{b\}\\
&  \iff\operatorname{Cl}(S_{\gamma}^{\prime}\cap\mathbb{Q}^{n})=\bigcap
\limits_{\sigma\in\Gamma}(S_{\Re\sigma_{\ast}(\gamma)}^{\prime}\cap
\Lambda_{\Im\sigma_{\ast}(\gamma)}^{\prime})\text{,}%
\end{align*}
\emph{which and Proposition \ref{both2} and the fact that \ref{paisolution2}
holds imply that }%
\[
\operatorname{Cl}(S_{\gamma}^{\prime}\cap\mathbb{Q}^{n})=\bigcap
\limits_{\sigma\in\Gamma}(S_{\Re\sigma_{\ast}(\gamma)}^{\prime}\cap
\Lambda_{\Im\sigma_{\ast}(\gamma)}^{\prime})
\]
\emph{holds. \ref{completeanswer} is the case }$q(x)\equiv\left\vert
x\right\vert ^{2}$\emph{\ of this formula.}
\end{remark}

\begin{remark}
\emph{The entire paper remains valid if we replace \textquotedblleft let
}$\Omega$\emph{\ be the algebraic closure of }$\mathbb{Q}(\gamma_{1}%
,\ldots,\gamma_{n})$\emph{\ in }$\mathbb{C}$\emph{\textquotedblright\ by
\textquotedblleft let }$\Omega$\emph{\ be an algebraically closed subfield of
}$\mathbb{C}$ \emph{containing }$\mathbb{Q}(\gamma_{1},\ldots,\gamma_{n}%
)$\emph{\textquotedblright. }
\end{remark}

\begin{remark}
\emph{Let }$k$\emph{\ be a subfield of }$\mathbb{R}$\emph{. Then the entire
paper remains valid if we replace }$\mathbb{Q}$\emph{\ by }$k$
\emph{everywhere.}
\end{remark}

Jun-ichi Matsushita

Contact address: 6-5-15, Higashiogu, Arakawa-ku, Tokyo 116-0012, Japan

E-mail: j-matsu@pk.highway.ne.jp

\end{document}